\documentclass{article}

\usepackage{amsfonts}
\usepackage{amssymb}
\usepackage{amsmath}
\usepackage{amsthm}
\usepackage{color}

\theoremstyle{plain}
\newtheorem{theorem}{Theorem}[section]
\newtheorem{lemma}[theorem]{Lemma}
\newtheorem{corollary}[theorem]{Corollary}

\theoremstyle{definition}
\newtheorem{definition}[theorem]{Definition}
\newtheorem{example}[theorem]{Example}

\def\ds{\displaystyle}

\newcommand{\RR}{{\mathbb R}}
\newcommand{\NN}{{\mathbb N}}
\newcommand{\SF}{{\mathcal S}}
\newcommand{\UU}{{\mathcal U}}

\begin{document}

\title{The Daniell Integral}
\author{\small{Elliot Blackstone and Piotr Mikusi\'nski}\\
\small{University of Central Florida, Orlando, Florida}}
\date{}

\maketitle

\section{Introduction}

In the following development of the Daniell Integral, we do not use the standard approach of introducing auxiliary spaces of the ``over-functions" and ``under-functions" (see, for example,  \cite{driver03}, \cite{Loomis}, or \cite{Royden}).   These spaces are only used as a step in the construction and are not needed afterwards.  Instead, we consider a construction of the Daniell integral that is modeled after the following definition of the Lebesgue integrable functions:

A real function $f$, defined on $\RR^N$, is Lebesgue integrable if there exists a sequence of simple functions $f_1,f_2,f_3, ...$ such that
\begin{enumerate}
\item[$\mathbb{A}$] $ \sum\limits_{n=1}^{\infty} \int |f_n| < \infty $
\item[$\mathbb{B}$] $ f(x) = \sum\limits_{n=1}^{\infty} f_n(x)$ for every $x \in \RR^N$ for which $\sum\limits_{n=1}^{\infty} |f_n(x)| < \infty $.
\end{enumerate}

This approach to the Lebesgue integral has been introduced in \cite{JM} and used \cite{Hilbert}, \cite{JMPM}, and \cite{PMMT}. It gives a very fast and natural way of developing the theory of the Lebesgue integral as well as the Bochner integral.  In this article we use this approach to the Daniell integral.  This method allows us to introduce the integral and the space of integrable functions in one step without any other constructions.  Moreover, our approach simplifies the proofs of important theorems on the integral.  The construction of a complete Daniell space, based on this method, has been presented in \cite{mik89}.

\section{Daniell Spaces}

\begin{definition}\label{def1.2.1}
A {\it Riesz space}, or {\it vector lattice}, is a vector space which is closed under the operations $\max(f,g)=f\lor g$ and $\min(f,g)=f\wedge g$.
\end{definition}

Notice that if $f$ is in some Riesz space, then $|f|$ is also in that Riesz space.

\begin{definition}\label{def1.2.2}
A triple $(X, \mathcal{U}, \int)$ is called a {\it Daniell space} if $X$ is a nonempty set, $\mathcal{U}$ is a Riesz space of real valued functions on $X$, and $\int: \mathcal{U} \to \RR$ is a linear functional such that
\begin{enumerate}
\item[I] $\int f \geq 0$ whenever $f\geq 0$,
\item[II] $\int f_n \to 0$ for every non-increasing sequence of functions $f_n \in \mathcal{U}$ such that $f_n(x) \to 0$ for every $x\in X$.
\end{enumerate}
\end{definition}

\begin{example}\label{l1}
 Let $X=\NN$ and let $\mathcal{U}$ be the space of all real valued functions $f$ on $\NN$ that are $0$ at all but finitely many $n$. Define
 $$
 \int f = \sum_{n=1}^\infty f(n).
 $$
It is easy to see that this is a Daniell space.
\end{example}

\begin{example}\label{Leb}
 By a {\it semi-open interval in} $\mathbb R^N$ we mean a set
$I$ which can be represented as 
\begin{equation*}
I=[ a_1,b_1) \times \dots  \times [a_N,b_N).
\end{equation*}
In other words, $x=(x_1,\dots ,x_N)\in I$ if $a_k\leq x_k<b_k$ for $k=1,2,\dots ,N$. The collection of all semi-open intervals in $\RR^N$ will be denoted by $\mathfrak{I}(\RR^N)$. 

For an arbitrary interval $I$ in $\RR^N$ (not necessarily semi-open) by $\mu(I)$ we mean the $N$-dimensional volume of $I$. If $N=1$, then $\mu(I)$ is just the length of $I$; if $N=2$ then $\mu(I)$ is the area of $I$, if $N=3$ then $\mu(I)$ is the volume of $I$, and so on.

By a {\it simple function} we mean a finite linear combination of characteristic functions of semi-open intervals:

\begin{equation}\label{eq2.14.2}
f=\lambda_1\chi_{I_1}+\dots  +\lambda_n\chi_{I_n},
\end{equation}
where $\lambda_i \in\RR$.
For the simple function $f$ in \eqref{eq2.14.2} define
\begin{equation*}
 \int f=\lambda_1 \mu(I_1)+\dots  +\lambda_n \mu(I_n).
\end{equation*}
The space of all simple functions on $\RR^N$ will be denoted by $\SF(\RR^N)$. We will show that $(\RR^N, \SF(\RR^N), \int)$ is a Daniell space. It is clear that $\SF(\RR^N)$ is a Riesz space and that condition I is satisfied.  It remains to show that II holds. First we need the following lemma.

\begin{lemma}\label{lem2.2.2} Let $I_1, I_2,\dots \in \mathfrak{I}(\RR^N)$ 
 be a partition of an interval $I\in \mathfrak{I}(\RR^N)$, i.e., the intervals 
$I_1, I_2,\dots $ are disjoint and $\bigcup^\infty_{n=1}I_n = I$. Then
\begin{equation*}
\sum^\infty_{n=1}\mu(I_n) = \mu(I).
\end{equation*}
\end{lemma}

\begin{proof} Clearly, $\sum^\infty_{n=1}\mu(I_n) \leq \mu(I)$. Suppose $\sum^\infty_{n=1}\mu(I_n) < \mu(I)-\varepsilon$
for some $\varepsilon >0$. There exist numbers $\varepsilon_n >0$ such that
\begin{equation*}
\sum^\infty_{n=1} \mu(I_n+B_{\varepsilon_n}) < \mu(I)-\frac{\varepsilon}{2},
\end{equation*}
where $B_\varepsilon = \{x\in\RR^N: \|x\|<\varepsilon \}$. Let $J\subset I$ be a compact interval in $\RR^N$ such that
\[
\mu(I) < \mu(J)+\frac{\varepsilon}{2}.
\]
Since the sets $I_n+B_{\varepsilon_n}$ are open and $
J \subset \bigcup_{n=1}^\infty (I_n+B_{\varepsilon_n})$,
we have
$$
J \subset \bigcup_{n=1}^m (I_n+B_{\varepsilon_n})
$$
for some $m\in\NN$. But then
$$
\mu(I) < \mu(J)+\frac{\varepsilon}{2} \leq  \sum_{n=1}^m \mu(I_n+B_{\varepsilon_n})+\frac{\varepsilon}{2} < \sum_{n=1}^\infty \mu(I_n+B_{\varepsilon_n})+\frac{\varepsilon}{2} <  \mu(I).
$$
\end{proof}

Now we are ready to prove that condition II is satisfied.

\begin{theorem} \label{thm2.2.2}  Let $(f_n)$ be a non-increasing sequence of non-negative simple functions such that $\lim_{n \rightarrow \infty}f_n(x) = 0$ for every $x \in \mathbb R^N$. Then
$\lim_{n \rightarrow \infty}\int f_n = 0$. 
\end{theorem}

\begin{proof}  Since the sequence $( \int f_n )$ is non-increasing and bounded from below (by $0$), it converges. Let
\begin{equation}\label{eq2.2.5}
\lim_{n \rightarrow \infty}\int f_n  = \varepsilon.
\end{equation}
Suppose $\varepsilon > 0$.  Let $I\in\mathfrak{I}(\RR^N)$ be an interval containing the
support of $f_1$ (and thus the support of every $f_n$, $n = 1, 2, \dots $  ). Let $\alpha = \frac{\varepsilon}{2\mu(I)}$. For $n = 1, 2, \dots $ define
$$
A_n = \{ x \in I :f_n(x) < \alpha \}
\;\; \text{and} \;\;
B_1 = A_1,\ B_n = A_n \setminus A_{n-1} \;\;\text{for} \;\; n \geq 2.
$$ 
Note that $B_n$'s are disjoint (because $A_{n-1}\subseteq A_n$) and 
$\cup_{n=1}^{\infty} B_n = I$ (because $\cup_{n=1}^{\infty} A_n = I$ ). Since $f_n$'s are simple functions, $B_n$'s are finite unions of disjoint semi-open intervals, say
$$
B_n  =  I_{n,1} \cup \dots \cup I_{n,k_n}.
$$  
The intervals
$$
I_{1,1},\dots , 
I_{1,k_1},\dots , 
I_{n,1},\dots , 
I_{n,k_n}, \dots 
$$
satisfy the assumptions of Lemma \ref{lem2.2.2} and thus
$$
\sum^\infty_{n=1}
\sum^{k_n}_{k=1} \mu(I_{n,k})  = \mu(I). 
$$
Let $n_0 \in \mathbb N$ be such that
\begin{equation}\label{eq2.2.6}
\sum^\infty_{n=n_0+1}
\sum^{k_n}_{k=1}\mu(I_{n,k})  < \delta,
\end{equation}
where $\ds \delta = \frac{\varepsilon}{2 \max |f_1|}$ . Set
$B =B_1 \cup  \dots  \cup B_{n_0}$ and define 
 two auxiliary functions $g$ and $h$:
$$
g(x)= \begin{cases} 
f_{n_0}(x) & \text{ for } x \in B,\\
0 & \text{ otherwise},
\end{cases} 
$$
and 
$$
h(x)=\begin{cases} 
0 & \text{ for } x \in B,\\
f_{n_0}(x) & \text{ otherwise}. 
\end{cases} 
$$ 
Since $B \subset A_{n_0}$, we have $f_{n_0}(x) < \alpha$ for $x\in B$. 
Consequently, $g(x) < \alpha$ for all $x \in \mathbb R^N$, which gives us
\begin{equation}\label{eq2.2.7}
\int g < \alpha \mu (I) = \frac{\varepsilon}{2}.
\end{equation}
Moreover, 
\begin{equation}\label{eq2.2.8}
\int h < \delta \max |f_{n_0}| \leq \delta \max |f_1| 
= \frac{\varepsilon}{2} ,
\end{equation}
because of \eqref{eq2.2.6}. As $f_{n_0} = g + h$, we have
$$
\int f_{n_0} = \int g + \int h < \varepsilon,
$$
by \eqref{eq2.2.7} and \eqref{eq2.2.8}. Since the sequence $\{\int f_n\}$ is non-increasing, we conclude
$$
\lim_{n \rightarrow \infty}\int f_n  \leq 
\int f_{n_0}  <  \varepsilon,
$$
which contradicts \eqref{eq2.2.5}. Therefore $\varepsilon = 0$, which completes the proof. 
\end{proof}

\end{example}

In Section \ref{DS&M} we generalize Example \ref{Leb} to abstract measure spaces.

\begin{definition}\label{def1.2.3}
Let $f$ be a real function on $X$.  If there exist functions $f_n \in \mathcal{U}, n \in \mathbb{N}$, such that
\begin{enumerate}
\item[$\mathbb{A}$] $\sum\limits_{n=1}^{\infty} \int |f_n| < \infty,$
\item[$\mathbb{B}$] $f(x) = \sum\limits_{n=1}^{\infty} f_n(x)$ for every $x \in X$ for which $\sum\limits_{n=1}^{\infty} |f_n(x)| < \infty,$
\end{enumerate}
then we write $f \simeq \sum\limits_{n=1}^{\infty}f_n$ or $f \simeq f_1 + f_2 + f_3 + \cdots$.
\end{definition}

\begin{definition}\label{def1.2.4}
A Daniell space (X, $\mathcal{U}$, $\int$) will be called $\it{complete}$ if  $f \simeq \sum\limits_{n=1}^{\infty}f_n$, for some $f_1,f_2, \dots \in \mathcal{U}$, implies that $f \in \mathcal{U}$.
\end{definition}

The space in Example \ref{Leb} is an example of a Daniell space that is not complete. Consider, for example, the function 
$$
f=\sum_{n=1}^\infty \frac1{2^n} \chi_{[n-1,n)} .
$$

In the next section we show that every Daniell space can be extended to a complete Daniell space.  It will later become clear that our unusual definition of completeness is equivalent to completeness in normed spaces. The problem is that we cannot simply say that completeness means that $\sum\limits_{n=1}^{\infty} \int |f_n| < \infty$ implies that $\sum\limits_{n=1}^{\infty} f_n \in \mathcal{U}$, because the series need not converge at every point, so the function $\sum\limits_{n=1}^{\infty} f_n$ is not well defined.

\section{Extension of Daniell spaces}

\begin{definition}\label{def1.2.5}
Given a Daniell space $(X, \mathcal{U}, \int)$, let $\mathcal{U}^*$ be the space of all real valued functions $f$ on $X$ for which there exists a sequence of functions $f_1,f_2, \ldots \in \mathcal{U}$ such that $f \simeq \sum\limits_{n=1}^{\infty}f_n$.  
\end{definition}

It is our goal to show that $(X, \mathcal{U}^*, \int)$ is a complete Daniell space, where 
the integral of $f \simeq \sum\limits_{n=1}^{\infty}f_n$ is defined as $\int f = \sum\limits_{n=1}^{\infty} \int f_n$. First we need to show that the integral is  independent of a particular representation of $f$.  This will require a couple of technical lemmas. 

\begin{lemma}\label{lem1.2.5}
If the sequences $(g_n)$ and $(h_n)$, $g_n,h_n \in \mathcal{U}$, are non-decreasing and $\lim\limits_{n \to \infty} h_n(x) \leq \lim\limits_{n \to \infty} g_n(x)$ for every $x \in X$, then $\lim\limits_{n \to \infty}\int h_n(x) \leq \lim\limits_{n \to \infty} \int g_n(x)$.
\end{lemma}

\begin{proof} Fix $k \in \mathbb{N}$.  Since the functions $h_k-(h_k\wedge g_n)$, $n\in \mathbb{N}$, form a non-increasing sequence which converges to zero at ever point of $X$, we have 
\begin{equation}
\lim_{n\to \infty}\int h_k - \int (h_k\wedge g_n) = 0 \nonumber
\end{equation}
and hence
\begin{equation}
\int h_k = \lim_{n\to \infty}\int (h_k\wedge g_n) \leq \lim_{n\to \infty}\int g_n. \nonumber
\end{equation}
Since this is true for all $k$, we let $k\to \infty$, and obtain the desired inequality.
\end{proof}

\begin{lemma}\label{thm1.2.6}
If $f \simeq \sum\limits_{n=1}^{\infty} f_n$ and $f \geq 0$, then $\sum\limits_{n=1}^{\infty} \int f_n \geq 0$.
\end{lemma}

\begin{proof}First, note that condition $\mathbb{A}$ of Definition \ref{def1.2.3} ensures the convergence of $\sum\limits_{n=1}^{\infty} \int f_n$.  To show this sum is greater than or equal to 0, we begin by fixing some $p \in \mathbb{N}$.  Then, for $n \in \mathbb{N}$, define
\begin{equation}
g_n=f_1 + f_2+ \cdots + f_p + |f_{p+1}| + \cdots + |f_{p+n}| \text{ and } h_n=g_n\wedge 0. \nonumber
\end{equation}

The sequences $(g_n)$ and $(h_n)$ are non-decreasing, $g_n,h_n \in \mathcal{U}$ and $\lim\limits_{n \to \infty} g_n =  \lim\limits_{n \to \infty} h_n$ (possibly $\infty$).  The equality of the limits follows from $f \geq 0$ and condition $\mathbb{B}$ from Definition \ref{def1.2.3}.  Thus, by Lemma \ref{lem1.2.5}, we have $ \lim\limits_{n \to \infty} \int g_n =  \lim\limits_{n \to \infty} \int h_n \geq 0$.  So,
\begin{equation}
\int f_1 + \int f_2+ \cdots + \int f_p + \int |f_{p+1}| + \int |f_{p+2}| + \cdots \geq 0 \nonumber
\end{equation}
So by letting $p \to \infty$, we obtain $\sum\limits_{n=1}^{\infty} \int f_n \geq 0$.
\end{proof}

Now we obtain the desired result as a corollary.

\begin{corollary}\label{cor1.2.7}
If $f \simeq \sum\limits_{n=1}^{\infty} f_n$ and $f \simeq \sum\limits_{n=1}^{\infty} g_n$, then $\sum\limits_{n=1}^{\infty} \int f_n = \sum\limits_{n=1}^{\infty} \int g_n$.
\end{corollary}

\begin{proof} Since $0 \simeq f_1-g_1+f_2-g_2+\cdots$, we have
\begin{equation}
\sum_{n=1}^{\infty}\int f_n-\sum_{n=1}^{\infty}\int g_n \geq 0. \nonumber
\end{equation}
Similarly, we have 
\begin{equation}
\sum_{n=1}^{\infty}\int g_n-\sum_{n=1}^{\infty}\int f_n \geq 0, \nonumber
\end{equation}
which proves $\sum\limits_{n=1}^{\infty}\int g_n=\sum\limits_{n=1}^{\infty}\int f_n$.
\end{proof}

Formally, we should distinguish between the integral of a function in $\mathcal{U}$ and in $\mathcal{U}^*$.  It turns out to be unnecessary since, if $f\in \mathcal{U}$, then $f\in \mathcal{U}^*$ and both integrals are the same. Indeed, it suffices to observe that for $f \in \mathcal{U}$ we have $f \simeq f +0+0+\cdots$.

\begin{corollary}\label{cor1.2.9}
$\mathcal{U}^*$ is a vector space and $\int $ is a linear functional on $\mathcal{U}^*$.  Moreover, if $f,g \in \mathcal{U}^*$ and $f\leq g$, then $\int f\leq \int g$.
\end{corollary}

\begin{proof}
   If  $f_n \simeq \sum\limits_{n=1}^{\infty} f_n$, $g \simeq \sum\limits_{n=1}^{\infty} g_n$ and $\lambda \in \RR$, then 
\begin{equation}
f+g=f_1+g_1+f_2+g_2+\cdots \quad \text{and} \quad \lambda f \simeq \lambda f_1+\lambda f_2 + \cdots. \nonumber
\end{equation}
Consequently, 
\begin{equation}
\int (f+g) = \int f+\int g \text{ and } \int \lambda f = \lambda \int f. \nonumber
\end{equation}

If $f,g \in \mathcal{U}^*$ and $f\leq g$, then $g-f\in \mathcal{U}^*$ and $g-f \geq 0$.  Hence $\int (g-f) \geq 0$ by Lemma \ref{thm1.2.6}, therefore, $\int f \leq \int g$.
\end{proof}

\begin{theorem}\label{thm1.2.10}
If $f\in \mathcal{U}^*$, then $|f|\in \mathcal{U}^*$ and $|\int f|\leq \int |f|$.  Moreover, if $f\simeq \sum\limits_{n=1}^{\infty}f_n$, then $\int |f|=\lim\limits_{n\to \infty}\int |f_1 + \cdots +f_n|$.
\end{theorem}

\begin{proof} Let $f\simeq \sum\limits_{n=1}^{\infty}f_n$.  Define
\begin{equation}
A = \{ x\in X : \sum_{n=1}^{\infty}|f_n| < \infty \} \text{ and } s_n = f_1 +\cdots + f_n. \nonumber
\end{equation}
Then, $f(x)=\lim\limits_{n\to \infty} s_n(x)$ for all $x\in A$.  In other words,
\begin{equation}\label{eq3.6.1}
|f|= |s_1(x)|+(|s_2(x)|-|s_1(x)|)+(|s_3(x)|-|s_2(x)|)+\cdots \text{ for } x\in A.
\end{equation}
Let $g_1=|s_1|$ and $g_n=|s_n|-|s_{n-1}|$ for $n\geq 2$. 
We claim that 
$$
|f|\simeq g_1+f_1-f_1+g_2+f_2-f_2+\cdots.
$$  
We will show that $\sum\limits_{n=1}^{\infty} \int |g_n| < \infty$ and that $|f(x)|=\sum\limits_{n=1}^{\infty} g_n$ for all $x\in A$.

First, for $n\geq 2$, we have
\begin{equation}
|g_n| = ||s_n|-|s_{n-1}|| \leq |s_n-s_{n-1}| = |f_n| \nonumber
\end{equation}
Thus, $\sum\limits_{n=1}^{\infty} \int |g_n| \leq \sum\limits_{n=1}^{\infty} \int |f_n| < \infty$, by Corollary \ref{cor1.2.9}, and since $f\simeq \sum\limits_{n=1}^{\infty}f_n$.  Next, by \eqref{eq3.6.1} we have that $|f(x)|=\sum\limits_{n=1}^{\infty} g_n$ for all $x\in A$.  If $x\notin A$, then the sum is not absolutely convergent.  Therefore, $|f|\in \mathcal{U}^*$.

Since $f\leq |f|$ and $-f\leq |f|$, we have $\int  f \leq \int  |f|$ and $-\int  f\leq \int  |f|$ by Corollary \ref{cor1.2.9}.  Thus, $|\int  f|\leq \int  |f|$.

Lastly, we have
\begin{equation}
\int |f|=\sum_{n=1}^{\infty}\int g_n=\lim_{n\to \infty}\int |s_n|=\lim_{n\to \infty}\int |f_1+\cdots +f_n|. \nonumber
\end{equation}
\end{proof}

\begin{corollary}\label{cor1.2.11}
$\mathcal{U}^*$ is closed under lattice operations.
\end{corollary}

\begin{proof}For $f,g\in \mathcal{U}^*$,
\begin{equation}
f\lor g=\frac{1}{2}(f+g+|f-g|)\text{,  } f \wedge g=\frac{1}{2}(f+g-|f-g|).\nonumber
\end{equation}
These two identities, the fact that $\mathcal{U}^*$ is a vector space (Corollary \ref{cor1.2.9}) and Theorem \ref{thm1.2.10} gives our proof.
\end{proof}

Now we address the question of completeness of $\mathcal{U}^*$. The following lemma is a crucial step in that direction.

\begin{lemma}\label{lem1.2.12}
If $f\in \mathcal{U}^*$, then for every $\varepsilon >0$ there exists a sequence of functions $f_1,f_2,\ldots \in \mathcal{U}$ such that $f\simeq \sum\limits_{n=1}^{\infty}f_n$ and $\sum\limits_{n=1}^{\infty}\int |f_n| \leq \int  f+\varepsilon$.
\end{lemma}

\begin{proof} Let $\varepsilon >0$ be given and let $f\simeq\sum\limits_{n=1}^{\infty}g_n$.  Choose $n_1\in \mathbb{N}$ such that $\sum\limits_{n_1+1}^{\infty}\int |g_n|<\frac{\varepsilon}{2}$.  By Theorem \ref{thm1.2.10}, we have $\int |f|=\lim\limits_{n\to\infty}|g_1+\cdots +g_n|$, so there exists an $n_2\in\mathbb{N}$ such that 

\begin{equation}
\int |g_1+\cdots +g_n| < \int |f|+\frac{\varepsilon}{2} \nonumber
\end{equation}

for every $n\geq n_2$.  Let $n_0=\max(n_1,n_2)$ and define $f_1=g_1+\cdots +g_{n_0}$, $f_n=g_{n_0+n-1}$ for $n\geq 2$.  Then, $f\simeq \sum\limits_{n=1}^{\infty}f_n$ and 

\begin{equation}
\sum_{n=1}^{\infty}\int |f_n|=\int |g_1+\cdots +g_{n_0}|+\sum_{n_0+1}^{\infty}\int |g_n|\leq \int |f|+\frac{\varepsilon}{2}+\frac{\varepsilon}{2},\nonumber
\end{equation}

which completes our proof.
\end{proof}

\begin{theorem}\label{thm1.2.13}
If $f\simeq \sum\limits_{n=1}^{\infty}f_n$ with $f_n\in \mathcal{U}^*$, then $f\in \mathcal{U}^*$ and $\int f=\sum\limits_{n=1}^{\infty}\int f_n$.
\end{theorem}

\begin{proof} Let $f\simeq \sum\limits_{n=1}^{\infty}f_n$ with $f_n\in \mathcal{U}^*$.  Choose $g_{i_n}\in \mathcal{U}$, $i,n\in \NN$, such that 

\begin{equation}
f_i\simeq \sum_{n=1}^{\infty}g_{i_n}\text{ and } \sum_{n=1}^{\infty}\int |g_{i_n}|\leq \int  |f_i|+2^{-i}\text{ for } i=1,2,\ldots \nonumber
\end{equation}

Let $(h_n)$ be a sequence arranged from all the functions $g_{i_n}$.  Then clearly \\ $f\simeq \sum\limits_{n=1}^{\infty}h_n$ which implies $f\in \mathcal{U}^*$ and $\int f=\sum\limits_{n=1}^{\infty}\int f_n$.
\end{proof}

\begin{corollary}\label{cor1.2.14}
For every non-increasing sequence of functions $f_n\in\mathcal{U}^*$ such that $f_n(x)\to 0$ for every $x\in\mathcal{X}$, we have $\int f_n \to 0$.
\end{corollary}

\begin{proof} The observation of $0\simeq f_1+(f_2-f_1)+(f_3-f_2)+\cdots$ combine with Theorem \ref{thm1.2.13} gives our proof.
\end{proof}

\begin{corollary} \label{cor2.5.1} Let $f_1, f_2, \dots \in  \mathcal{U}^*$. If  $\sum^\infty_{n=1}\int|f_n|  <  \infty$, then there exists $f\in \mathcal{U}^*$ such that $f  \simeq  f_1  +  f_2  + \dots $ . 
\end{corollary}

\begin{proof}  The function $f$ can be defined as follows:
$$
f(x)=\begin{cases} 
\sum^\infty_{n=1}f_n(x) & \text{ whenever }
\sum^\infty_{n=1}|f_n(x)|  <  \infty,\\
0 & \text{ otherwise.}
\end{cases}  
$$
\end{proof}

\begin{theorem}\label{thm1.2.15}
Every Daniell space $(X,\mathcal{U},\int)$ can be extended to a complete Daniell space $(X,\mathcal{U}^*,\int )$.
\end{theorem}

\begin{proof} To show $(X,\mathcal{U}^*,\int )$ is a Daniell space, we need to satisfy conditions I and II from Definition \ref{def1.2.2}.  Both of these conditions are satisfied as a result of Theorem \ref{thm1.2.6} and Corollaries \ref{cor1.2.9}, \ref{cor1.2.11} and \ref{cor1.2.14}.  A direct result of Theorem \ref{thm1.2.13} shows that  $(X,\mathcal{U}^*,\int )$ is complete.
\end{proof}

 From our construction and the definition of completeness it is clear that $(X,\mathcal{U}^*,\int )$ is the smallest complete extension of $(X,\mathcal{U},\int )$.  
 
In the remainder of this article we assume that $(X,\mathcal{U},\int )$ is a complete Daniell space.

\section{Norm in a Daniell space}

\begin{definition}[Norm in $\UU$] \label{def2.6.1} The functional
$\|\cdot\| : \UU  \rightarrow  \mathbb R$
defined by $\|f\|  =  \int |f|$ is called the {\it norm} in $(X,\mathcal{U},\int )$.
\end{definition}

The functional $\|\cdot\|$ is well-defined in view of Theorem \ref{thm1.2.10}.  By Corollary \ref{cor1.2.9},
we have 
$$
\|\lambda f\|  =  \int|\lambda f|  =  \int|\lambda||f|  =  
|\lambda|\int|f|  =  |\lambda|\|f\|.
$$
Since $|f  +  g|  \leq  |f|  +  |g|$, we have
$$
\|f  +  g\|  =  \int|f  +  g|  \leq  \int|f|  +  \int|g|  =  \|f\|  +  \|g\|,
$$
by the same corollary.  However, $\|\cdot\|$ need not be a norm since, in general, $\|f\|=0$ does not imply $f=0$. 

\begin{definition}[Null function] \label{def2.6.2}  A function $f\in \UU$ is called a {\it null function} if $\int|f| = 0$. 
\end{definition}

\begin{theorem} \label{thm2.6.1} If $f\in \UU$ is a null function and $|g|  \leq  |f|$, then $g\in \UU$ and $g$ is a null function.
\end{theorem}

\begin{proof}  Note that 
\begin{equation}\label{eq2.6.1}
g  \simeq  |f|  +  |f|  + \dots 
\end{equation}
In fact, since $f$ is a null function, we have 
$\int|f|  +  \int|f|  + \dots  = 0  +  0  + \dots   <  \infty$.
Moreover, if the series $f(x) + f(x) + \dots $ is absolutely
convergent at some $x \in  X$, then $f(x) = 0$. 
But then $g(x) = 0$, and we have $g(x) = |f(x)|  +  |f(x)|  + \dots $. 
This proves that \eqref{eq2.6.1} holds, and thus $g\in \UU$. Clearly, $g$ is a null function.
\end{proof}

\begin{definition}
 Functions $f,g\in\UU$ are called {\it equivalent} if $f-g$ is a null function.
\end{definition}
 
 It is easy to check that the defined relation is an equivalence in $\UU$.

Let ${\mathfrak U}$ be the space of equivalence classes in $\UU$. The equivalence class of $f \in \UU$ is denoted by $[f]$,
i.e.,
$$
[ f ] = \left\{g  \in  \UU : \int|f  -  g| = 0\right\}.
$$
It is easy to check that 
$$
[ f ]  +  [ g ] = [ f  +  g ], \quad
\lambda [ f ] = [\lambda f ], \quad 
\|[ f ]\| = \int|f|
$$
are well defined and that $({\mathfrak U}, \|\cdot\|)$ is
a normed space.  We will show later that it is a Banach space.

In practice, we often do not distinguish between $\UU$ and ${\mathfrak U}$ and formulate everything in terms of $\UU$. We also refer to $\|\cdot\|$ in $\UU$ as a norm.  This abuse of language does not lead to any problems as long as we remember what it means. 
 
\begin{definition} \label{def2.6.3} {\rm (Convergence in norm)} We say that a sequence of functions $f_1, f_2, \dots \in \UU$
converges to a function $f \in  \UU$ {\it in norm},
denoted by $f_n \rightarrow  f$ i.n., if $\|f_n-f\|  \rightarrow  0$.
\end{definition} 

As a convergence defined by a norm, it has the following properties:

{\it
If $f_n  \rightarrow  f$ i.n. and
$\lambda  \in  \mathbb R$, then 
$\lambda f_n  \rightarrow  \lambda f$ i.n.

If $f_n  \rightarrow  f$ i.n. and $g_n  \rightarrow  g$ i.n.,
then $f_n + g_n  \rightarrow  f + g$ i.n.

\noindent
Moreover

If $f_n  \rightarrow  f$ i.n., then $|f_n|  \rightarrow  |f|$ i.n.,}

\bigskip

\noindent
which follows immediately from the inequality 
$$\left| |f_n|  -  |f| \right|  \leq  |f_n  -  f|.$$

\begin{theorem} \label{thm2.6.2} If $f_n  \rightarrow  f$ i.n., then 
$\int f_n  \rightarrow  \int f$.
\end{theorem}

\begin{proof} 
$|\int f_n-\int f| = |\int(f_n-f)|  \leq  \int|f_n-f| 
\rightarrow  0$. 
\end{proof}

\begin{theorem} \label{thm2.6.3} If $f  \simeq  f_1  +  f_2  + \dots $ , then the series $f_1  +  f_2  + \dots $ converges to $f$ in norm. 
\end{theorem}

\begin{proof}  Let $\varepsilon > 0$. There exists an integer $n_0$ such that  $\sum^\infty_{n=n_0}\int|f_n| < \varepsilon$. Since
$$
f - f_1 - \dots   -  f_n \simeq f_{n+1}  +  f_{n+2}  +  \dots  ,
$$
for every $n > n_0$, we have
$$
\int|f  -  f_1  - \dots  -  f_n| \leq
\int|f_{n+1}|  +  \int|f_{n+2}|  + \dots   <  \varepsilon,
$$
by Theorem \ref{thm1.2.10}.  
\end{proof}

\section{Convergence Almost Everywhere}

If $f \simeq f_1 + f_2 + \dots $ then the series $f_1(x) + f_2(x) + \dots $ need not converge at every $x \in X$.  

\begin{definition}[Null set] \label{def2.7.1} A set $S \subseteq X$ is called a {\it null set} if its characteristic function is a null function.
\end{definition} 

\begin{theorem}\label{thm2.7.1}
 A subset of a null set is a null set. A countable union of null sets is a null set.
\end{theorem}
 
\begin{proof}
 The fact that a subset of a null set is a null set is a direct consequence of Theorem \ref{thm2.6.1}.  If $A_1,A_2, \dots \subset X$ are null sets, then there is $f\in\UU$ such that
 $$
 f \simeq \chi_{A_1}+\chi_{A_2}+\dots,  
 $$
 by Corollary \ref{cor2.5.1}. Since $\chi_{A_1 \cup A_2 \cup \dots} \leq f $ and $f$ is a null function, $A_1 \cup A_2 \cup \dots$ is a null set by Theorem \ref{thm2.6.1}.
\end{proof}

\begin{definition}[Equality almost everywhere] \label{def2.7.2}  Let $f,g : X \to \RR$. If the set of all
$x \in X$ for which $f(x) \neq g(x)$ is a null set,
then we say that {\it $f$ equals $g$ almost everywhere} and write $f = g$ a.e.. 
\end{definition} 

\begin{theorem} \label{thm2.7.2} $f = g$ a.e. if and only if $\int|f-g| = 0$.   
\end{theorem}

\begin{proof}  Let $h$ be the characteristic function of the set $Z$ of all $x \in X$ for which $f(x) \neq g(x)$.
  
If $f = g$ a.e., then $\int|h| = \int h = 0$. 
Therefore 
$$
|f - g| \simeq h + h + \dots  ,
$$
which implies $\int|f-g| = 0$. 

Conversely, if $\int|f-g| = 0$, then
$$
h \simeq |f-g| + |f-g| + \dots  ,
$$
and hence $\int h = 0$.  This shows that $Z$ is a null set, i.e.,
$f = g$ a.e. 
 
\end{proof}

\begin{corollary}\quad
\begin{itemize}
\item[(a)] $f$ is a null function if and only if $f = 0$ a.e.  
\item[(b)] $f$ and $g$ are equivalent if and only if $f = g$ a.e. 
\end{itemize} 
\end{corollary}

Note that we do not need to know the value of a function at every point in order to find its integral.  It is sufficient to know the values almost everywhere, i.e., everywhere except a null set.  The function need not even be defined at every point.

\begin{theorem} \label{thm2.7.3}  Suppose $f_n \rightarrow f$ i.n. Then $f_n \rightarrow g$  i.n. if and only if $f = g$ a.e..
\end{theorem}

\begin{proof}  If $f_n \rightarrow f$ i.n. and $f = g$ a.e., then
$$
\|f_n-g\| = \int|f_n-g| \leq \int|f_n-f| + 
\int|f-g| = \int|f_n-f| = \|f_n-f\| 
\rightarrow 0.
$$

If $f_n \rightarrow f$ i.n. and $f_n \rightarrow g$ i.n., then 
$f_n-f_n \rightarrow f-g$ i.n..  This implies  
$$
\int|f-g|  =  \int|f_n-f_n-f + g|  \rightarrow  0,
$$
completing the proof.
\end{proof}

\begin{definition} \label{convae} {\rm (Convergence almost everywhere)} We say that a sequence of functions $f_1,  f_2,\dots $, defined on $X$  converges to $f$ {\it almost everywhere}, denoted by $f_n \rightarrow f$ a.e., if $f_n(x) \rightarrow f(x)$ for every $x$ except a null set.
\end{definition} 

Convergence almost everywhere has properties similar to convergence in norm.

{\it
If $ f_n \rightarrow f$ a.e. and
$\lambda \in \mathbb R$, then 
$\lambda f_n \rightarrow \lambda f$ a.e.

If $f_n \rightarrow f$ a.e. and $g_n \rightarrow g$ a.e.,
then  $f_n + g_n \rightarrow f + g$ a.e.

If $f_n \rightarrow f$ a.e., then $|f_n| \rightarrow |f|$
a.e.
}

\begin{theorem} \label{thm2.7.4} Suppose $f_n \rightarrow f$ a.e. Then $f_n \rightarrow g$ a.e. if and only if $f = g$ a.e.
\end{theorem}

\begin{proof}  If $f_n \rightarrow f$ a.e. and $f_n \rightarrow g$ a.e., then $f_n-f_n \rightarrow f-g$ a.e., which means
that $f-g = 0$ a.e.

Now, denote by $A$ the set of all $x \in X$ such that the sequence $(f_n(x))$ does not converge to $f(x)$ and by $B$ the set of all $x \in X$ such that $f(x) \neq g(x)$.  Then $A$ and $B$ are null sets and so is $A\cup B$. Since $f_n(x)\rightarrow g(x)$ for every $x$ not in $A\cup B$, we have  $f_n \rightarrow g$ a.e.
\end{proof}

\begin{theorem} \label{thm2.7.5} Let $f_1, f_2, \dots \in \UU$ and $\int|f_1| + \int|f_2| + \dots  < \infty$. Then the series 
$f_1 + f_2 + \dots $ converges almost everywhere.
\end{theorem}

\begin{proof}  By Corollary \ref{cor2.5.1}, there exists a function 
$f \in\mathcal{U}$ such that $f \simeq f_1 + f_2 + \dots $. Since $f(x) = \sum^\infty_{n=1}f_n(x)$ for every $x $ such that $\sum^\infty_{n=1}|f_n(x)| < \infty$, it suffices to show that the set of all points $x \in X$ for which the series
$\sum^\infty_{n=1}|f_n(x)|$ is not absolutely convergent is a null set.  Let $g$ be the characteristic function of that set. Then $g \simeq f_1 - f_1 + f_2 - f_2 + \dots  $, and consequently
$$
\int|g| = \int g =\int f_1 - \int f_1 + \int f_2 - \int f_2 + \dots  = 0. 
$$
\end{proof}

\begin{corollary} \label{cor2.7.1} If $f \simeq f_1 + f_2 + \dots $ , then 
$f = f_1 + f_2 + \dots $ a.e. 
\end{corollary}

\begin{theorem} \label{thm2.7.6}  Let $f_1, f_2, \dots  \in\mathcal{U}$  and $\int|f_1| + \int|f_2| + \dots  < \infty$.  Then  
$f = f_1 + f_2 + \dots $ a.e. if and only if 
$f = f_1 + f_2 + \dots $ i.n.
\end{theorem}

\begin{proof}  By Corollary \ref{cor2.5.1}, there exists a function 
$g \in\mathcal{U}$ such that $g \simeq f_1 + f_2 + \dots $ Then, by Theorem \ref{thm2.6.3} we have $g = f_1 + f_2 + \dots $ i.n. and, Corollary \ref{cor2.7.1},  we have $g = f_1 + f_2 + \dots $  a.e. 

Now, if $f = f_1 + f_2 + \dots $ a.e., then $f = g$ a.e., by Theorem \ref{thm2.7.4}. Hence $f = f_1 + f_2 + \dots $  i.n., by
Theorem \ref{thm2.7.3}.

Conversely, if $f = f_1 + f_2 + \dots $  i.n., then  $f = g$ a.e., by Theorem \ref{thm2.7.3}. Hence $f = f_1 + f_2 + \dots $  a.e., by
Theorem \ref{thm2.7.4}.
\end{proof}

\section{Fundamental Convergence Theorems}

Now we are ready to justify our definition of completeness of Daniell spaces.

\begin{theorem} \label{thm2.8.1} The space $({\mathfrak U}, \|\cdot\|)$ is a complete normed space.
\end{theorem}

\begin{proof}   We will prove that every absolutely convergent series in $\UU$ converges in norm. Let $f_n \in \UU$, $n = 1, 2,\dots $, and  let $\sum^\infty_{n=1}\int|f_n| < \infty$. Then, by Corollary \ref{cor2.5.1}, there exists an $f \in \UU$ such that $f \simeq \sum^\infty_{n=1}f_n$. This in turn implies, by Theorem \ref{thm2.6.3}, that the series $\sum^\infty_{n=1}f_n$ converges to $f$ in norm, proving the theorem.
\end{proof}

\begin{theorem} \label{thm2.8.2} If $f_n \rightarrow f$ i.n. in $\UU$, then there exists a subsequence $(f_{p_n})$  of $(f_n)$ such that $f_{p_n} \rightarrow f$ a.e.
\end{theorem}

\begin{proof}  Since $\int|f_n-f| \rightarrow 0$, there exists an increasing sequence of positive integers $(p_n)$ such that 
$\int|f_{p_n}-f| < 2^{-n}$.  Then
$$
\int|f_{p_{n+1}}-f_{p_n}| \leq \int|f_{p_{n+1}}-f| + \int|f-f_{p_n}| <
\frac{3}{2^{n+1}}
$$
and consequently
$$
\int|f_{p_1}| + \int|f_{p_2}-f_{p_1}| + \int|f_{p_3}-f_{p_2}| + \dots   < \infty.
$$
Thus, there exists a $g \in \UU$ such that 
$$
g \simeq f_{p_1} + (f_{p_2}-f_{p_1}) + (f_{p_3}-f_{p_2}) + \dots ,
$$
and, by Corollary \ref{cor2.7.1},
$$
g = f_{p_1} + (f_{p_2}-f_{p_1}) + (f_{ p_3}-f_{p_2}) + \dots  \text{ a.e.} \; .
$$
This means $f_{p_n} \rightarrow g$ a.e.  Since also
$f_{p_n} \rightarrow g$ i.n.  and $f_{p_n} \rightarrow f$ i.n., we conclude $f = g$ a.e., by Theorem \ref{thm2.7.3}.  Therefore  $f_{p_n} \rightarrow f$ a.e., by Theorem \ref{thm2.7.4}.  
\end{proof}

A sequence of functions is called {\it monotone} if it is non-increasing or non-decreasing. 

\begin{theorem} \label{thm2.8.3} {\rm (Monotone Convergence Theorem)} If $f_n\in\UU$ is a monotone sequence and
$\left|\int f_n\right| \leq M$ for some constant $M$ and all $n \in \mathbb N$, then there exists $f\in\UU$ such that $f_n \rightarrow f$ i.n. and  $f_n \rightarrow f$ a.e. Moreover, $\left|\int f \right| \leq M$.
\end{theorem}

\begin{proof}  Without loss of generality, we can assume that the sequence is non-decreasing and the functions are non-negative.  In such a case
$$
\int|f_1| + \int|f_2-f_1| + \dots  + \int|f_n-f_{n-1}| =
\int|f_n| \leq M,
$$
for every $n \in \mathbb N$.  By letting $n \rightarrow \infty$, we obtain
$$
\int|f_1| + \int|f_2-f_1| + \dots  \leq M.
$$
By Corollary \ref{cor2.5.1}, there exists an $f \in \UU$
such that $f \simeq f_1 + (f_2-f_1) + \dots $. Hence,  
$f_n \rightarrow f$ i.n., by
Theorem \ref{thm2.6.3}, and  $f_n \rightarrow f$ a.e., by Corollary \ref{cor2.7.1}. Finally
\begin{align*}
\left|\int f \right| &= \left|\int f _1 + 
\int (f_2-f _1)+\int (f_3-f_2)- \dots  \right|\\ 
&\leq  \int|f_1| + \int|f_2- f_1| +
\int|f_3- f_2| + \dots  \leq M. 
\end{align*}
\end{proof}

\begin{theorem}  \label{thm2.8.4} {\rm (Dominated Convergence Theorem)} If a sequence of functions $f_n\in\UU$ converges almost everywhere to a function $f$ and there exists a function $h\in\UU$ such that $|f_n| \leq h$ for every $n \in \mathbb N$, then 
$f\in\UU$ and $f_n \rightarrow f$ i.n.
\end{theorem}

\begin{proof}  For $m,n = 1, 2,\dots $, define
$$
g_{m,n} = \max \{|f_m|,\dots ,|f_{m+n}|\}.
$$
Then, for every fixed $m \in \mathbb N$, the sequence 
$(g_{m,1}, g_{m,2},\dots )$  is non-decreasing and, since
$$
\left |\int g_{m,n} \right| = \int g_{m,n} \leq 
\int h < \infty,
$$
there is $g_m\in\UU$ such that $g_{m,n} \rightarrow g_m$   a.e., as $n \rightarrow \infty$. 

Note that the sequence $(g_n)$ is non-increasing and $0 \leq g_n$ for all $n \in \mathbb N$. Thus, it converges to a function $g$ at every point and, by the Monotone Convergence Theorem, $g\in\UU$ and $g_n \rightarrow g$ i.n.  Now we consider two cases.

Case 1: Suppose $f = 0$. Then $f_n \rightarrow 0$ a.e., and
therefore $g_n \rightarrow 0$ a.e. Since the sequence $(f_n)$ converges in norm, we obtain $g_n \rightarrow 0$ i.n.  Hence
$$
\int|f_n| \leq \int g_n \rightarrow 0, 
$$
which proves the theorem in the first case.

Case 2: When $f$ is an arbitrary function, then for every increasing sequence of positive integers $(p_n)$ we have
$$
h_n = f_{p_{n+1}}-f_{p_n} \rightarrow 0 \text{ a.e.}
$$
and $|h_n| \leq 2h$ for every $n \in \mathbb N$.  By Case 1, we must have $h_n \rightarrow 0$ i.n. This
shows that the sequence $(f_n)$ is a Cauchy sequence in
$\UU$ and therefore it converges in norm to some
$\tilde f \in \UU$, by Theorem \ref{thm2.8.1}.  On the other
hand, by Theorem \ref{thm2.8.2}, there exists an increasing sequence of positive integers $q_n$ such that $f_{q_n} \rightarrow \tilde f$ a.e. But
$f_{q_n} \rightarrow f$ a.e., and thus $\tilde f = f$ a.e.  This, in
view of Theorem \ref{thm2.7.3}, implies that $f_n \rightarrow f$ i.n. 
\end{proof}

\begin{theorem}  \label{thm2.8.5} {\rm (Fatou's Lemma)} Let $f_n\in\UU$ be a sequence of non-negative functions such that $\int f_n \leq M$  for some $M$ and every $n \in \mathbb N$.  If $f_n \to f$ a.e., then $f\in\UU$ and $\int f \leq M$.
\end{theorem}

\begin{proof}  Let $\varphi_{n,k} = f_n \wedge  f_{n + 1}\wedge \dots \wedge f_{n + k}$, for $n,k \in \mathbb N$. For a fixed $n \in \mathbb N$ the sequence  
$(\varphi_{n,1}, \varphi_{n,2},\dots )$ is a decreasing sequence in $\UU$ such that $\left|\int\varphi_{n,k}\right| \leq \int\varphi_{n,1} < \infty$. Thus, by the Monotone Convergence Theorem, it converges almost everywhere to a function $\varphi_n\in\UU$.  Thus we have
$$
\varphi_n = \inf \{f_n, f_{n + 1},f_{n + 2},\dots \} \text{ a.e.}
$$
Since $\int \varphi_n \leq \int f_n \leq M$ and 
$\varphi_1 \leq \varphi_2 \leq \varphi_3 \leq \dots $,
the sequence $(\varphi_n)$ converges almost everywhere to a
function $g\in\UU$ and we have $\int g \leq M$, again by the Monotone Convergence Theorem. If $f_n(x)\rightarrow f(x)$ for some $x\in X$, then $\varphi_n(x)\rightarrow f(x)$.  Thus $f=g$ a.e. and $\int f \leq M$. 
\end{proof}

\section{Daniell Spaces and Measures}\label{DS&M}

Let $(X, \Sigma, \mu)$ be a measure space. A function $f:X \to \RR$ is called a simple function if it has the form
$$
f=\alpha_1 \chi_{A_1} + \dots + \alpha_n \chi_{A_n}
$$
where $n\in \NN$, $A_1,\dots , A_n \in \Sigma$, and $\alpha_1, \dots , \alpha_n \in \RR$.  The space of all simple function on the measure space $(X, \Sigma, \mu)$ will be denoted by $\SF(X, \Sigma, \mu)$.  On $\SF(X, \Sigma, \mu)$ we define a functional $\int$:
$$
\int_X f\, d\mu =\int_X (\alpha_1 \chi_{A_1} + \dots + \alpha_n \chi_{A_n}) d\mu = \alpha_1 \mu(A_1) + \dots + \alpha_n \mu(A_n).
$$

A standard argument shows that the integral of a simple function is well defined.

\begin{theorem}
 $\left(X,\SF(X, \Sigma, \mu), \int \right)$ is a Daniell space.
\end{theorem}

\begin{proof} We will only prove that $\lim_{n\to \infty} \int_X f_n\, d\mu=0$  for any  non-increasing sequence of simple functions $(f_n)$ convergent to $0$ at every point of $X$.
 
 Let $(f_n)$ be a non-increasing sequence of simple functions convergent to $0$ at every point of $X$ and let $\varepsilon >0$. Let
 $$
 A_n = \left\{x\in X : f_n(x)> \frac{\varepsilon}{ 2\mu ( {\rm supp} f_1)}  \right\}.
 $$   
Then $A_n\in\Sigma$ and, since $(f_n)$ is a non-increasing sequence, we have $A_1 \supset A_2 \supset \dots $. Define $B_n=A_n\setminus A_{n+1}$. Since $\lim_{n\to\infty} f_n(x)=0$ for all $x\in X$, we have $A_1=\cup_{n=1}^\infty B_n$. Then $\mu(A_1)=\sum_{n=1}^\infty \mu(B_n)$, by $\sigma$-additivity of $\mu$. Hence
 $\sum_{k=n}^\infty \mu(B_k)\to 0$ as $n\to \infty$.  Let $n_0\in\NN$ be such that 
 $$
 \mu(A_n) = \sum_{k=n}^\infty \mu(B_k) < \frac{\varepsilon}{2\|f_1\|_\infty} 
 $$
 for all $n>n_0$. Then, for $n>n_0$, we have
 \begin{align*}
 \int_X f_n\, d\mu &= \int_{A_n} f_n\, d\mu +\int_{X\setminus A_n} f_n\, d\mu \\
 &\leq \|f_n\|_\infty \mu (A_n) + \frac{\varepsilon}{2\mu ( {\rm supp} f_1)} \mu ( {\rm supp} f_n) < \varepsilon.
 \end{align*}
\end{proof}

Now let $(X,\mathcal{U},\int )$ be a complete Daniell space.  A subset $A\subset X$ is called integrable if $\chi_A\in\UU$. A subset $A\subset X$ is called measurable if $A\cap B$ is integrable for every integrable $B\subset X$. Let $\Sigma_{(X,\mathcal{U},\int )}$ be the collection of all measurable subsets of $X$.  Then we define a set function $\mu_{\int} : \Sigma_{(X,\mathcal{U},\int )} \to [0,\infty]$:
$$
\mu_{\int}(A)= \begin{cases} 
\int \chi_A & \text{ if $A$ is integrable},\\
\infty & \text{ otherwise}.
\end{cases} 
$$

\begin{theorem}
 If $(X,\mathcal{U},\int )$ is a complete Daniell space, then $\Sigma_{(X,\mathcal{U},\int )}$ is a $\sigma$-algebra and $\mu_{\int}$ is a $\sigma$-additive measure.
\end{theorem}

\begin{proof}
 Let $S,U,V\in\Sigma_{(X,\mathcal{U},\int )}$ and let $S$ be integrable. Since
 $$
\chi_{ (U \setminus V) \cap S}=\chi_{U\cap S}-(\chi_{V\cap S}\wedge\chi_{U\cap S})
 $$
 and
 $$
\chi_{ (U \cup V) \cap S}=\chi_{U\cap S}+\chi_{U\cap S}-(\chi_{V\cap S}\wedge\chi_{U\cap S}),
 $$
$U\setminus V,U\cup V \in \Sigma_{(X,\mathcal{U},\int )}$.
 If $U_1, U_2, \dots \in \Sigma_{(X,\mathcal{U},\int )}$, then the sets
 $$
 V_1=U_1, V_n=U_n \setminus (U_1 \cup \dots \cup U_{n-1}) \text{ for } n>1
 $$
 are measurable in view of the above.  If $S\in\Sigma_{(X,\mathcal{U},\int )}$ is integrable, then
 $$
\chi_{ \left(\bigcup_{n=1}^\infty U_n  \right)\cap S} =\chi_{ \left(\bigcup_{n=1}^\infty V_n  \right)\cap S} \simeq \sum_{n=1}^\infty \chi_{V_n\cap S}
 $$
 and thus $\bigcup_{n=1}^\infty U_n \in \Sigma_{(X,\mathcal{U},\int )}$.

 If $U_1, U_2, \dots \in \Sigma_{(X,\mathcal{U},\int )}$ are disjoint and $\sum_{n=1}^\infty \mu_{\int}(U_n)<\infty$, then $\bigcup_{n=1}^\infty U_n \simeq \sum_{n=1}^\infty U_n$ and hence 
 $\mu_{\int}\left(\bigcup_{n=1}^\infty U_n\right)  = \sum_{n=1}^\infty \mu_{\int}(U_n)$. If $U_1, U_2, \dots \in \Sigma_{(X,\mathcal{U},\int )}$ are disjoint and $\bigcup_{n=1}^\infty \mu_{\int}(U_n)=\infty$, then it is easy to see that $\mu_{\int}\left(\bigcup_{n=1}^\infty U_n\right) =\infty$. 
\end{proof}

Since a complete Daniell space $(X,\mathcal{U},\int )$ defines a measure space, namely $(X, \Sigma_{(X,\mathcal{U},\int )}, \mu_{\int})$, we can consider the standard space of integrable functions $L^1(X, \Sigma_{(X,\mathcal{U},\int )}, \mu_{\int})$.  
Do we have $\UU=L^1(X, \Sigma_{(X,\mathcal{U},\int )}, \mu_{\int})$?  First we prove that $L^1(X, \Sigma_{(X,\mathcal{U},\int )}, \mu_{\int})\subset \UU $.

For our purpose it will be convenient to use the following definition of $L^1(X,\Sigma, \mu)$.

\begin{definition}[$L^1(X,\Sigma, \mu)$] Let $(X,\Sigma, \mu)$ be a measure space.  By $L^1(X,\Sigma, \mu)$ we mean the space of all functions $f:X\to \RR$ such that there are simple functions $f_1,f_2, \ldots \in \SF(X, \Sigma, \mu)$ satisfying the following conditions
\begin{enumerate}
\item[$\mathbb{A}$] $ \sum\limits_{n=1}^{\infty} \int |f_n|< \infty $,
\item[$\mathbb{B}$] $ f(x) = \sum\limits_{n=1}^{\infty}f_n(x)$ for every $x \in X$ for which $ \sum\limits_{n=1}^{\infty} |f_n(x)| < \infty $.
\end{enumerate}
\end{definition}

\begin{theorem}
If $(X,\mathcal{U},\int )$ is a complete Daniell space, then $L^1(X, \Sigma_{(X,\mathcal{U},\int )}, \mu_{\int})\subset \UU $.
\end{theorem}

\begin{proof}
If $f\in L^1(X, \Sigma_{(X,\mathcal{U},\int )}, \mu_{\int})$, then $f\simeq f_1+f_2+\cdots$, where each $f_n$ is a linear combination of characteristic functions of sets of finite measure.  Consequently, there are $A_1, A_2, \ldots \in \Sigma_{(X,\mathcal{U},\int )}$ and $\alpha_1, \alpha_2, \ldots \in \RR$ such that
$f\simeq \alpha_1\chi_{A_1}+ \alpha_2\chi_{A_2}+\dots$. Since, by definition, $\chi_A\in \UU$ for every integrable $A\in \Sigma_{(X,\mathcal{U},\int )}$, we conclude that $f\in\UU$, by completeness of $\UU$.
\end{proof}

It turns that  $\UU\subset L^1(X, \Sigma_{(X,\mathcal{U},\int )}, \mu_{\int})$ need not hold in general.  It does hold under an additional natural condition on $\UU$, called Stone's condition.

\begin{definition}[Stone's Condition]\label{def1.6.2}
  If $f \in \mathcal{U}$, then $f\wedge 1 \in \mathcal{U}$.
\end{definition}

Note that, since $\UU$ is a vector space, we can replace $1$ with any other positive constant.

\begin{theorem}
If $(X,\mathcal{U},\int )$ is a complete Daniell space satisfying Stone's condition, then $\UU\subset L^1(X, \Sigma_{(X,\mathcal{U},\int )}, \mu_{\int})$.
\end{theorem}

\begin{proof}

Let $f\in \mathcal{U}$. Without loss of generality, we may assume that f is non-negative (otherwise, we take positive and negative parts).  For $\alpha >0$ let $S_\alpha = \{x\in X : f(x)>\alpha \}$.  Define
$$
g=f-f\wedge \alpha \quad \text{and}\quad  g_n=(ng)\wedge 1. 
$$
Notice that both $g$ and $g_n$ are in $\mathcal{U}$ for all $n\in \NN$ since $\mathcal{U}$ is a vector space that satisfies Stone's condition.  We can see that $g_n \to \chi_{\mathcal{S}_\alpha}$ as $n\to\infty$ at all points of $X$. Since $g_n \leq f/\alpha$ for all $n\in\NN$ and $f/\alpha \in \mathcal{U}$, we conclude $\chi_{S_\alpha}\in \mathcal{U}$, by the Monotone Convergence Theorem.

Now we assume that $0\leq f\leq M$, where $M$ is some positive constant.  For a fixed $n\in \NN$ we let 
\begin{equation}
A_k =\left \{x\in X:f(x)>k \frac{M}{2^n}\right\},\; k=0,1,\ldots,2^n.\nonumber
\end{equation}
From our previous argument, we know that $\chi_{A_k}\in \mathcal{U}$ for $k=0,1,\ldots,2^n$.  Also, notice that $A_0\supset A_1 \supset \cdots A_{2^n-1}\supset A_{2^n}=\emptyset$.  Next, we define $B_k=A_{k-1}\backslash A_k$.  This implies that $\chi_{B_k}\in \mathcal{U}$, since $\chi_{B_k}=\chi_{A_{k-1}}-\chi_{A_k}$, for $k=1,2,\ldots,2^n$, and that 
$$
{\rm supp}(f)=\bigcup_{k=1}^{2^n}B_k.
$$
Now we can define $f_n$ as
\begin{equation}
f_n=\sum_{k=1}^{2^n}(k-1)\frac{M}{2^n}\chi_{B_k}. \nonumber
\end{equation}
Since $f_n$ is a simple function, we have $f_n\in L^1(X,\Sigma,\mu)$.  By construction of $f_n$, we can see that $(f_n)$ is an increasing sequence and $f_n\to f$ at all points of $X$.  We claim that $f\simeq g_1 + g_2 + \cdots$, where $g_1 = f_1$ and $g_n=f_n-f_{n-1}$ for $n=2,3,\ldots$.  Since the sequence $(f_n)$ is increasing, we have 
\begin{equation}
\sum_{n=1}^{\infty}\int|g_n|=\sum_{n=1}^{\infty}\int g_n = \lim_{n\to \infty}\int f_n = \int  f < \infty. \nonumber
\end{equation}
Also,
\begin{equation}
f(x)=\sum_{n=1}^{\infty}g_n(x)=\lim_{n\to \infty}f_n(x) \text{ when } \sum_{n=1}^{\infty}|g_n(x)|=\lim_{n\to \infty}f_n(x)<\infty. \nonumber
\end{equation}
Therefore, $f\simeq f_1 + (f_2-f_1) + \cdots$, which means $f\in L^1(X,\Sigma,\mu)$. 

If $f$ is an unbounded nonnegative function, we consider functions $f_n=f\wedge n$, where $n\in\NN$. By the previous argument, $f_n\in  L^1(X,\Sigma,\mu)$ for all $n\in\NN$.  Then we can show that $f\simeq f_1 +(f_2-f_1)+ \cdots$ as above. Thus $f\in  L^1(X,\Sigma,\mu)$.  
\end{proof}

\end{document}